\documentclass[11pt,a4paper]{article}
	\usepackage[utf8]{inputenc}
	\usepackage{amsthm}
	\usepackage{amssymb,amsmath}
  \usepackage[dvips]{graphicx}
	\usepackage{pstricks}
	\usepackage{xcolor}
	\usepackage{float}
	\usepackage{multicol}
	\theoremstyle{plain}
	\newtheorem{theorem}{Theorem}[section]
	
	\newtheorem{lemma}[theorem]{Lemma}
	\newtheorem{corol}[theorem]{Corollary}

	\theoremstyle{definition}
	
	\newtheorem{remark}[theorem]{Remark}
	\newtheorem{example}[theorem]{Example}

	\newcommand{\N}{{{\mathbb N}}}

	\newcommand{\cC}{{\mathcal C}}
\begin{document}

	\title{\textbf{Positive solutions for second order boundary value problems with sign changing Green's functions}}
		\date{}
		\author{Alberto Cabada$^1$\footnote{Partially supported by Ministerio de Econom\'{\i}a y Competitividad, Spain, and FEDER, project MTM2013-43014-P, and by the Agencia Estatal de Investigaci\'on (AEI) of Spain under grant MTM2016-75140-P, co-financed by the European Community fund FEDER.}, Ricardo Engui\c{c}a$^2$\footnote{Partially supported by Fundação para a Ci\^{e}ncia e a Tecnologia, Portugal, UID/MAT/04561/2013} and Luc\'{\i}a L\'opez-Somoza$^1$\footnote{
				Partially supported by the Agencia Estatal de Investigaci\'on (AEI) of Spain under grant MTM2016-75140-P, co-financed by the European Community fund FEDER.}\\
				$^1$ Instituto de Matem\'aticas, Facultade de Matem\'aticas,\\ Universidade de Santiago de Compostela, 15782,\\  Santiago de Compostela, Galicia, Spain\\ 	e-mail: alberto.cabada@usc.es, \; lucia.lopez.somoza@usc.es\\
	$^2$  Departamento de Matem\'atica, Instituto Polit\'ecnico de Lisboa,\\ Lisboa, Portugal \\
e-mail: rroque@adm.isel.pt  }
		\maketitle

	\begin{abstract}
	 In this paper we analyse some possibilities of finding positive solutions for second order boundary value problems with Dirichlet and periodic boundary conditions, for which the correspondent Green's functions change sign. The obtained results can also be adapted to the Neumann and Mixed boundary conditions.
		
		\vspace{.5cm}

				
	\end{abstract}
	
	\section{Introduction}
	
	In the literature it has been widely studied the existence of positive solutions for boundary value problems (BVP), namely second order BVP with Periodic and Dirichlet boundary conditions. A standard technique consists on obtaining the existence of positive solutions  through Krasnoselskii's fixed point theorem on cones, or to use fixed point index theory. In these cases, the positivity of the associated Green's functions is usually fundamental to prove such results. In this paper we are able to prove existence of solutions for several problems where the associated Green's function changes sign.
	
	Hill's operator properties have been described in several papers, where existence and multiplicity results, comparison principles, Green's functions and spectral analysis were studied. Some of these results can be originally found in \cite{cacidAAA, cacid, cacilu,  torres, zhang}.
	
	Positivity results for BVP where the Green's function can vanish are treated for example in \cite{gkw1, webb}. In \cite{gkw1}, Graef, Kong and Wang studied the periodic BVP (with $T=1$ in the paper)
	\begin{equation*}\begin{split}
	&u''(t)+a(t)\,u(t)=g(t)\,f(u(t)), \quad t\in(0,T),\\
	&u(0)=u(T), \ u'(0)=u'(T),
	\end{split}	\end{equation*}
	with $f$ and $g$ nonnegative continuous functions and $g$ satisfying the condition $\min_{t\in[0,1]}g(t)>0$. They assumed the Green's function to be nonnegative and to satisfy the following condition
	\begin{equation}\label{int-graef}
	\min_{0\le s\le T} \int_0^T G(t,s) \, dt>0.
	\end{equation}
	In \cite{webb}, Webb considered weaker assumptions to prove the existence of positive solutions of the previous problem, but he still assumed the Green's function to be nonnegative. Despite our results do not require the Green's function to be nonnegative, they could be applied to this particular case, obtaining positive solutions assuming an integral condition weaker than \eqref{int-graef} (see Remarks \ref{R-int-Graef} and \ref{R-int-Graef2} in Section 3).
	
	On the other hand, some existence results for BVP with sign-changing Green's function have been considered in \cite{cabinftoj,infpietoj}, where the authors asked for the existence of a subinterval $[c,d]\subset [0,T]$, a function $\phi\in L^1([0,T])$ and a constant $c\in(0,1]$ such that the Green's function $G$ satisfies the following condition:
	\begin{equation}\label{G-banda}\begin{split}
		&|G(t,s)|\le \phi(s) \text{ for all } t\in[0,T] \text{ and almost every } s\in[0,T],\\
	&G(t,s)\ge c\, \phi(s) \text{ for all } t\in[c,d] \text{ and almost every } s\in[0,T].
	\end{split}\end{equation}

	It must be pointed out that, if we consider a periodic problem with constant potential $a(t)=\rho^2$ for which the related Green's function changes its sign  (i.e. $\rho > \pi/T$, $\rho \neq 2 k \pi/T$, $k=1,2, \ldots$), condition \eqref{G-banda} is never  fulfilled for any strictly positive function $\phi$. This is due to the fact that in such situation the Green's function is constant along the straight lines of slope equals to one (see \cite{cabada1, cabada2} for details). Meanwhile, as we will prove on Section \ref{sect_ex_per}, our results can be applied without further complications for this  case. 
	
	Moreover, for Dirichlet BVP with constant potential $a(t)=\rho^2$ with sign change Green's function (i.e. $\rho > \pi/T$, $\rho \neq k \pi/T$, $k=1,2, \ldots$), as a direct consequence of expression \eqref{e-G-Dir}, it is immediate to verify that condition \eqref{G-banda} holds if and only if $\rho^2$ lies between the first and the second eigenvalues of the problem ($\frac{\pi}{T}<\rho<\frac{2\,\pi}{T}$) but it is never satisfied for $\rho>\frac{2\,\pi}{T}$. However, as we will point out in Section \ref{sect-Dir-constant}, our results can be applied for any nonresonant value of $\rho > \pi/T$. Despite this, we must note that the imposed restrictions increase with $\rho$.
	
	Furthermore, in  \cite{cabinftoj,infpietoj} the authors proved the existence of solutions in the cone
	\[K_0=\left\{u\in\mathcal{C}[0,T], \ \min_{t\in[c,d]} u(t)\ge c \|u\|\right\},\]
	that is, they ensured the positivity of the solutions on the subinterval $[c,d]$ but such solutions were allowed to change sign when considering the whole interval $[0,T]$. 
	
	As far as we know, positive solutions for BVP with sign-changing Green's function can be tracked only as back as 2011 in the papers \cite{ma, ZA}. In the first of these papers, R. Ma considers the following one parameter family of problems,
		\begin{equation}		\label{e-per-lambda}
	u''(t)+a(t)\,u(t)=\lambda\,g(t)\,f(u(t)), \; t\in(0,T),\quad
	u(0)=u(T), \ u'(0)=u'(T).
\end{equation}
By using the Schauder's fixed point Theorem, the author obtains the existence of a positive solution for sufficiently small values of $\lambda$. These existence results are not comparable with the ones we will obtain in this paper.
	On the second paper \cite{ZA}, S. Zhong and Y. An study the following autonomous periodic BVP, with constant potential $\rho \in (0, \frac{3 \,\pi}{2\, T}]$.
	\begin{equation}
	\label{e-per-Zhong}
	u''+\rho^2 u=f(u),\, t\in(0,T),\quad u(0)=u(T), \; u'(0)=u'(T).
	\end{equation}
In this case, it is very well known that the related Green's function $G_P(t,s) \ge 0$ for all $\rho \in (0, \frac{\pi}{T}]$ and it changes sign for $\rho \in (\frac{\pi}{T}, \frac{3 \,\pi}{2\, T}]$ (see \cite{cabada1, cabada2}).
With this, it can be defined the constant
	\begin{equation*}
	\delta=
	\left\{
	\begin{array}{lll}
	\infty &\mbox{if} & \rho \in (0, \frac{\pi}{T}],\\
	 \inf_{t\in I} \frac{\int_{0}^{T} G^+_P(t,s)\,ds}{\int_{0}^{T} G^-_P(t,s)\,\,ds}&\mbox{if} & \rho \in (\frac{\pi}{T}, \frac{3 \,\pi}{2\, T}]
	 \end{array}
	 \right.
	\end{equation*}
and using the Krasnoselskii's fixed point Theorem, the authors prove the following existence result:

\begin{theorem}\cite[Theorem 3]{ZA}
\label{t-ZA}
Suppose that the following assumptions are fulfilled:
\begin{enumerate}
\item[$(J1)$]  $f:[0, \infty) \to [0, \infty) $ is continuous.
\item[$(J2)$]   $0 \le m = \inf_{u \ge 0}{\{f(u)\}}$ and $M= \sup_{u \ge 0}{\{f(u)\}} \le M \le \infty$.
\item[$(J3)$]  $M/m \le \delta$, with $M/m = \infty$ when $m=0$.
\end{enumerate}
Moreover, if $\delta=\infty$ assume that
\[
 \lim_{x\to \infty} \,  \frac{f(x)}{x}< \rho^2 < \lim_{x\to 0^+} \frac{f(x)}{x}.
\]
Then problem \eqref{e-per-Zhong} has a positive solution on $[0,T]$.
\end{theorem}

\noindent	Concerning this specific case, along this paper we improve the range of the values $\rho$ for which the result is still valid. Furthermore, we apply our study to nonconstant potentials and nonautonomous nonlinear parts.
	
	As we will see, some of the positivity conditions imposed in the periodic BVP cannot be adapted for the Dirichlet BVP, so the approach that must be used needs to be considerably modified, by using, in this case, a different type of cones.

	The rest of the paper is divided in the following way: in Section \ref{sect-preliminaries} we state some preliminary results considering the Hill's operator, in Section \ref{sect-periodic} are proved some new results concerning the existence of a positive solution for the Hill's periodic BVP in the case that the Green's function may change sign. Moreover, on this section, such existence results are generalized to other boundary conditions. In Section  \ref{sect_ex_per} we improve Theorem \ref{t-ZA} for the periodic problem with constant potential and in Section \ref{sect-Dir-constant} we approach the Dirichlet BVP, also in the case of constant potential, where as far as we know, no results for sign changing Green's function were proved before.

	\section{Preliminaries}
	\label{sect-preliminaries}
	Let $L[a]$ be the Hill's operator related to the potential  $a$
	\begin{equation*}
	L[a]\,u(t)\equiv u''(t)+ a(t)\,u(t), \quad t\in [0,T]\equiv I,
	\end{equation*}
	where $a\colon  I\rightarrow \mathbb{R}$, $a\in L^\alpha( I)$, $\alpha\ge 1$.
	
	Let $X \subset W^{2,1}( I)$ be a Banach space such that the homogeneous problem
	\begin{equation}
	\label{e-Hill-X}
	L[a]\,u(t)=0, \quad a.\,e. \ t\in I, \quad u\in X
	\end{equation}
	has only the trivial solution. This condition is known as operator $L[a]$ is nonresonant in $X$.
	
	Moreover, it is very well known that if this condition is satisfied and $\sigma \in L^1( I)$, the nonhomogeneous problem
	\begin{equation*}
	L[a]\,u(t)=\sigma(t), \quad a.\,e. \ t\in I, \quad u\in X
	\end{equation*}
	has a unique solution given by
	\begin{equation*}
	u(t)=\int_{0}^{T} G(t,s)\,\sigma(s)\, ds, \quad t\in I,
	\end{equation*}
	where $G$ is the corresponding Green's function.
	
	We denote $x \succ 0$ on $I$ if and only if  $x \ge 0$ on $I$ and $\int_0^T{x(s) \, ds } >0$.
	It is said that operator $L[a]$ satisfies a strong maximum principle (MP) in $X$ if and only if
	\begin{equation*}
	u\in X, \; L[a]\,u \succ 0 \text{ in }  I\Rightarrow u < 0 \text{ in }  (0,T).
	\end{equation*}
	Analogously, $L[a]$ satisfies the antimaximum principle (AMP) in $X$ if and only if
	\begin{equation*}
	u\in X, \; L[a]\,u \succ 0 \text{ in }  I\Rightarrow u > 0 \text{ in }   (0,T).
	\end{equation*}
	
	The next result is a direct consequence of \cite[Corollaries 1.6.6 and 1.6.12]{cabada2}, and it ensures that the maximum and anti-maximum principles for the periodic problem are equivalent to the constant sign of the Green's function.  
	
	\begin{lemma}
		\label{l-Green-comparison} 
				The following claims are equivalent:
		\begin{itemize}
			
			\item[(1)] The related Green's function $G$ of problem \eqref{e-Hill-X} satisfies $G (t,s) \ge 0$ $(\le 0)$ on $ I\times  I$.
			
			\item[(2)] Operator $L[a]$ satisfies a strong maximum (antimaximum) principle  in $X$.
			
		\end{itemize}
	\end{lemma}

	We will consider now the periodic boundary value problem
	\begin{equation}\label{e-P} \tag{$P$}\begin{split}
	 u''(t)+a(t)\,u(t)=0, \; t\in I,	\quad	u(0)=u(T),\; u'(0)=u'(T),
		\end{split}\end{equation}
			and we will denote its related Green's function as $G_P$.

	Now, let ${\lambda}_P$ be the smallest eigenvalue of the periodic problem
	$$
	u''(t)+(a(t) + \lambda)\, u(t)=0,\quad \mbox{ a. e.  }\;t\in I, \qquad u(0)=u(T),\; u'(0)=u'(T),
	$$
	and let ${\lambda}_A$ be the smallest eigenvalue of the anti-periodic problem
	$$
	u''(t)+(a(t) + \lambda)\, u(t)=0,\quad \mbox{ a. e.  }\;t\in I,
	\qquad u(0)=-u(T),\; u'(0)=-u'(T).
	$$
	In \cite{zhang} it is proved that  ${\lambda}_P<{\lambda}_A$.
	The following result relates the constant sign of the periodic Green's function with the sign of these eigenvalues:
	\begin{lemma} \cite[Theorem 1.1]{zhang}
		\label{l-zhang-eigen}
		Suppose that $a \in L^1( I)$, then:
		\begin{enumerate}
			\item $G_P(t,s) \le 0$ on $ I\times  I$ if and only if ${\lambda}_P>0$.
			\item $G_P(t,s) \ge 0$ on $ I\times  I$ if and only if ${\lambda}_P< 0\le {\lambda}_A$.
		\end{enumerate}
	\end{lemma}
	
If we consider other boundary value problems, such as Neumann problem,
	\begin{equation}\label{Neumann} \tag{$N$}\begin{split}
	 u''(t)+a(t)\,u(t)=0, \; t\in I,	\quad  u'(0)=u'(T)=0;
	\end{split}\end{equation}
	Dirichlet problem,
	\begin{equation}\label{Dirichlet} \tag{$D$}\begin{split}
	u''(t)+a(t)\,u(t)=0, \; t\in I,\quad u(0)=u(T)=0;
	\end{split}\end{equation}
	and Mixed problems
	\begin{equation}\label{Mixed1} \tag{$M_1$}\begin{split}
	u''(t)+a(t)\,u(t)=0, \; t\in I,\quad u'(0)=u(T)=0;
	\end{split}\end{equation}
	\begin{equation}\label{Mixed2} \tag{$M_2$}\begin{split} 
	u''(t)+a(t)\,u(t)=0, \;  t\in I,\quad u(0)=u'(T)=0;
	\end{split}\end{equation}
	 denoting $G_N$, $G_D$, $G_{M_1}$ and $G_{M_2}$ the related Green's functions and $\lambda_N$, $\lambda_D$, $\lambda_{M_1}$ and $\lambda_{M_2}$ the correspondent smallest eigenvalue of each of the problems, we know that the following results are satisfied (see \cite{cacilu}):
	
	\begin{lemma}
	\label{l-Meu-Dir-Mix}
	\begin{enumerate}
		\item $G_N(t,s)<0$ on $ I\times I$ if and only if $\lambda_N>0$.
		
		\item $G_N(t,s)\ge 0$ on $ I\times I$ if and only if $\lambda_N<0$, $\lambda_{M_1}\ge 0$ and $\lambda_{M_2}\ge 0$.
		
		\item $G_N$ changes sign if and only if $\min\{\lambda_{M_1},\,\lambda_{M_2}\}< 0$.
		
		\item $G_D(t,s)<0$ on $(0,T)\times(0,T)$ if and only if $\lambda_D>0$.
		
		\item $G_D$ changes sign if and only if $\lambda_D<0$.
		
		\item $G_{M_1}(t,s)<0$ on $[0,T)\times[0,T)$ if and only if $\lambda_{M_1}>0$.
		
		\item $G_{M_1}$ changes sign if and only if $\lambda_{M_1}<0$.
		
		\item $G_{M_2}(t,s)<0$ on $(0,T]\times(0,T]$ if and only if $\lambda_{M_2}>0$.
		
		\item $G_{M_2}$ changes sign if and only if $\lambda_{M_2}<0$.		
	\end{enumerate}
	\end{lemma}

	\section{Periodic boundary value problems}
	\label{sect-periodic}
	Consider now the following nonlinear and nonautonomous periodic boundary value problem
	\begin{equation}\label{hill_periodic}\begin{split}
	u''(t)+a(t)\,u(t)=f(t,u(t)), \; t\in I, \quad u(0)=u(T), \ u'(0)=u'(T).
	\end{split}\end{equation}
	
	We will assume that problem $(P)$ is nonresonant and ${\lambda}_A< 0$. From Lemma \ref{l-zhang-eigen}, it is clear that in this case the related Green's function changes its sign on $I \times I$.
	
	On the other hand, it is well-known that there exists $v_P$, a positive eigenfunction on $I$, unique up to a constant, related to $\lambda_P$, that is, $v_P$ is such that
	\begin{equation*}\begin{split}
	& v_P''(t)+a(t)\,v_P(t)=-\lambda_P\,v_P(t), \quad a.\,e. \ t\in I,\\
	& v_P(0)=v_P(T), \ v_P'(0)=v_P'(T).
	\end{split}\end{equation*}
	Therefore,
	\begin{equation*}
	v_P(t)=-\lambda_P\int_{0}^{T} G_P(t,s)\,v_P(s)\,ds
	\end{equation*}	
	and, since $v_P$ is positive and $\lambda_P<0$, we have that
	\begin{equation*}
	\int_{0}^{T} G_P(t,s)\,v_P(s)\,ds>0 \quad \forall\,t\in I
	\end{equation*}	
	and, consequently
	\begin{equation*}
	\int_{0}^{T} G^+_P(t,s)\,v_P(s)\,ds>\int_{0}^{T} G^-_P(t,s)\,v_P(s)\,ds  \quad \forall\,t\in I,
	\end{equation*}		
	where $G^+_P$ and $G^-_P$ are the positive and negative parts of $G_P$.
	
	Since the Green's function changes sign, it makes sense to define
	\begin{equation*}
	\gamma= \inf_{t\in I} \frac{\int_{0}^{T} G^+_P(t,s)\,v_P(s)\,ds}{\int_{0}^{T} G^-_P(t,s)\,v_P(s)\,ds}\, (>1) .
	\end{equation*}
	
	Moreover, in order to ensure the existence of solutions of problem \eqref{hill_periodic}, we will make the following assumptions:
	\begin{enumerate}
		\item[$(H_1)$] $f\colon I\times [0,\infty) \rightarrow [0,\infty)$ satisfies $L^1$-Carath\'eodory conditions, that is, $f(\cdot,u)$ is measurable for every $u\in\mathbb{R}$, $f(t,\cdot)$ is continuous for a.\,e. $t\in I$ and for each $r>0$ there exists $\phi_r \in L^1 (I)$ such that $f(t,u) \le \phi_r(t)$ for all $u\in[-r,r]$ and a.\,e. $t\in I$.
		
		\item[$(H_2)$] There exist two positive constants $m$ and $M$ such that $m\,v_P(t)\le f(t,x) \le M\,v_P(t)$ for every $t\in I$ and $x\ge 0$. Moreover, these constants satisfy that $\frac{M}{m}\le \gamma$.  
		
		\item[$(H_3)$] There exists $[c,d]\subset  I$ such that $\int_{c}^{d} G_P(t,s) dt \ge 0,$ for all $s\in  I$ and $\int_{c}^{d} G_P(t,s) dt > 0,$ for all $s\in  [c,d]$.
		
	\end{enumerate}
	
	\begin{remark}
	\label{r-condi-per}
	We note that condition $(H_2)$ includes, as particular cases, hypotheses $(J2)$ and $(J3)$ in Theorem \ref{t-ZA} imposed in \cite{ZA}. This is due to the fact that if $a(t)=\rho^2$, as in problem \eqref{e-per-Zhong}, we have that $\lambda_P=-\rho^2$ and $v_P(t)=1$ for all $t \in I$.
	Moreover, as we will point out in Section \ref{sect_ex_per}, we have that if $a(t)=\rho^2$ then
	\[\int_{0}^{T} {G_P(t,s)\,ds}=\frac{1}{\rho^2},\]
and condition $(H_3)$ is trivially fulfilled for $[c,d]=I$. 

 Moreover, we note that in $(H_2)$ we are not considering the possibility of $m=0$. Theorem \ref{t-ZA} includes this case, but only when $\gamma=+\infty$, which only happens when the Green's function is nonnegative. In \cite{ZA} the authors consider this possibility because they are assuming that $\rho\in \left( 0, \frac{3\,\pi}{2\,T}\right]$ and when $\rho\in \left( 0, \frac{\pi}{T}\right]$, $G_P$ is nonnegative. As we will see in Corollary \ref{cor-GP-posit}, hypothesis $(H_2)$ is not necessary in this case, so this is the reason why we do not consider the possibility $m=0$.
	\end{remark}
	
	We will consider the Banach space $(\mathcal{C}( I,\mathbb{R}), \, \|\cdot\|)$  coupled with the supremum norm $\|u\|\equiv \|u\|_\infty$, and define the cone 
	\[K=\left\{u\in \mathcal{C}( I,\mathbb{R}); \ u\ge 0 \, \mbox{on $I$}, \ \int_{0}^{T}u(s)\,ds \ge \sigma \|u\| \right\}, \]
	where 
	\begin{equation*}
	\sigma = \frac{\eta}{\displaystyle \max_{t,\,s\in I} \left\{ G_P(t,s)\right\}},
	\end{equation*}
being
	\begin{equation}
	\label{e-eta}
	\eta =\min_{s\in [c,d]}  \left\{\int_{c}^{d}G_P(t,s)\,dt\right\}.
	\end{equation}
	
	Now, it is clear that $u$ is a solution of the periodic problem \eqref{hill_periodic} if and only if it is a fixed point of the following operator
	\begin{equation*}
	{\mathcal{T}} u(t)=\int_{0}^{T} G_P(t,s)\,f(s,u(s)) \,ds.
	\end{equation*}
	
	\begin{lemma}
		Assume hypothesis $(H_1)-(H_3)$. Then ${\mathcal{T}} \colon \cC( I) \rightarrow \cC( I)$ is a completely continuous operator which maps the cone $K$ to itself.
	\end{lemma}
	
	\begin{proof}
	The proof that operator ${\mathcal{T}}$ is a completely continuous operator follows standard arguments and we omit it.
			
Let's see now that ${\mathcal{T}}$ maps the cone to itself. 
Considering $u \in K$, then, for all $t \in I$, the following inequalities are fulfilled:
			\begin{equation*}\begin{split}
			{\mathcal{T}} u(t)& =\int_{0}^{T}G_P(t,s)\,f(s,u(s))\,ds= \int_{0}^{T}\left(G^+_P(t,s)-G^-_P(t,s)\right)\,f(s,u(s))\,ds \\
			&\ge \int_{0}^{T}\left(m\, v_P(s)\,G^+_P(t,s)-M\,v_P(s)\,G^-_P(t,s)\right)\,ds \\ 
			&\ge m \left( \int_{0}^{T}G^+_P(t,s)\,v_P(s)\,ds-\gamma\int_{0}^{T}G^-_P(t,s)\,v_P(s)\,ds \right)\ge 0.
			\end{split}\end{equation*}
					Moreover,
			\begin{equation*}\begin{split}
			\int_{0}^{T} {\mathcal{T}} u(t) \,dt & \ge \int_{c}^{d} {\mathcal{T}} u(t) \,dt = \int_{c}^{d} \int_{0}^{T}G_P(t,s)\,f(s,u(s))\,ds \,dt =  \int_{0}^{T} f(s,u(s)) \int_{c}^{d} G_P(t,s)\,dt \,ds \\
			& \ge \eta \, \int_{0}^{T} f(s,u(s))\,ds,
			\end{split}\end{equation*}
			and since
			\begin{equation*}
			{\mathcal{T}} u(t) \le \max_{t,\,s\in I} \left\{G_P(t,s) \right\} \,\int_{0}^{T}f(s,u(s))\,ds,
			\end{equation*}
			we deduce that $\int_{0}^{T} {\mathcal{T}} u(t) \,dt  \ge \sigma \,{\mathcal{T}} u(t)$ for all $t\in I$, that is
			\begin{equation*}
			\int_{0}^{T} {\mathcal{T}} u(t) \,dt  \ge \sigma \,\|{\mathcal{T}} u\|,
			\end{equation*}
and the result is concluded.
	\end{proof}
	
	Now, in order to prove our existence results, as an immediate consequence of condition $(H_2)$, we deduce the following properties
	\begin{equation*}\begin{split}
	f_0&=\lim\limits_{x\rightarrow 0^+} \, \left\{\min_{t\in [c,d]} \frac{f(t,x)}{x}\right\}=\infty, \qquad f^{\infty}=\lim\limits_{x\rightarrow \infty} \, \left\{\max_{t\in I}  \frac{f(t,x)}{x}\right\}=0.
		\end{split}\end{equation*}
To this end, we will use some classical results regarding the fixed point index. We compile these results in the following lemma. Let $\Omega$ be an open bounded subset of a cone $K$ and let's denote $\bar \Omega$ and $\partial\Omega$ its closure and boundary, respectively. Moreover, let's denote $\Omega_K=\Omega \cap K$.
	
	\begin{lemma} \cite[Lemma 12.1]{amann2}
		Let $\Omega_K$ be an open bounded set with $0\in\Omega_K$ and $\bar \Omega_K\neq K$. Assume that $F\colon \bar \Omega_K \rightarrow K$ is a completely continuous map such that $x\neq Fx$ for all $x\in\partial \Omega_K$. Then the fixed point index $i_K(F,\Omega_K)$ has the following properties:
		\begin{enumerate}
			\item If there exists $e\in K\setminus\{0\}$ such that $x\neq Fx+\lambda\, e$ for all $x\in\partial \Omega_K$ and all $\lambda>0$, then $i_K(F,\Omega_K)=0$.
			\item If $x\neq \mu\, Fx$ for all $x\in\partial \Omega_K$ and for every  $\mu\le 1$, then $i_K(F,\Omega_K)=1.$
			\item If $i_K(F,\Omega_K)\neq 0$, then $F$ has a fixed point in $\Omega_K$.
			\item Let $\Omega^1_K$ be an open set with $\bar \Omega^1_K\subset \Omega_K$. If $i_K(F,\Omega_K)=1$ and $i_K(F,\Omega^1_K)=0$, then $F$ has a fixed point in $\Omega_K \setminus \bar \Omega^1_K$. The same result holds if $i_K(F,\Omega_K)=0$ and $i_K(F,\Omega^1_K)=1$.
		\end{enumerate}
	\end{lemma}
	
	Now we are in conditions to prove the existence results concerning the periodic problem \eqref{hill_periodic}  as follows.
	\begin{theorem}
		Assume that $\lambda_A<0$ and hypothesis $(H_1)-(H_3)$ hold. Then there exists at least one positive solution of problem \eqref{hill_periodic} in the cone $K$.
	\end{theorem}
	
	\begin{proof}
		
		Taking into account the definition of $f_0$, we know that there exists $\delta_1>0$ such that when $\|u\|\le\delta_1$, then
		$$f(t,u(t))> \,\frac{u(t)}{\eta}\,, \quad \forall \, t\in [c,d],$$
		with $\eta$ defined in \eqref{e-eta}.
		
		Let 
		$$\Omega_{1}=\{u\in K; \ \|u\|<\delta_1\}$$	
		and choose $u\in\partial \Omega_1$ and $e\in K\setminus\{0\}$. 
		
		We will prove that $u \neq {\mathcal{T}} u+\lambda\,e$ for every $\lambda>0$. 
		
		Assume, on the contrary, that there exists some $\lambda>0$ such that $u={\mathcal{T}} u+\lambda\,e$, that is,
		\begin{equation*}
		u(t)= {\mathcal{T}} u(t)+\lambda\,e(t)\ge {\mathcal{T}} u(t)  \quad \forall\,t\in I.
		\end{equation*}
		Then
		\begin{equation*}\begin{split}
		\int_{c}^{d}u(t)\, dt &\ge \int_{c}^{d} {\mathcal{T}} u(t)\,dt =\int_{c}^{d} \int_{0}^{T} G_P(t,s)\,f(s,u(s))\,ds\,dt =  \int_{0}^{T} \left(\int_{c}^{d} G_P(t,s) \,dt\right) \,f(s,u(s))\,ds  \\
		& \ge \int_{c}^{d} \left(\int_{c}^{d} G_P(t,s) \,dt\right) \,f(s,u(s))\,ds > \int_{c}^{d}u(s) \,ds,
		\end{split}\end{equation*}
		which is a contradiction. 
		
		Therefore we deduce that $i_K(T,\Omega_1)=0$.
		
		Now, we proceed in an analogous way to \cite{cacid, gkw1, gkw2}, we define
		 $\displaystyle \tilde{f}(t,u)=\max_{0\le z\le u}{f(t,z)}$. Clearly $\tilde{f}(t, \cdot)$ is a nondecreasing function on $[0,\infty)$. Moreover, since $f^{\infty}=0$ it is obvious that
$$\lim\limits_{x\rightarrow \infty} \, \left\{\max_{t\in I}  \frac{\tilde{f}(t,x)}{x}\right\}=0.$$

As a consequence, we know that there exists $\delta_2>0$ such that if $\|u\|\ge\delta_2$ then
		$$\tilde{f}(t,\|u\|)< \,\frac{\sigma^2}{T^2\, \eta}\,\|u\| \quad \forall \, t\in I.$$
		
		Let 
		$$\Omega_{2}=\{u\in K; \ \|u\|<\delta_2\}$$	
		and choose $u\in \partial\, \Omega_{2}$. 
		
		We will prove that $u \neq \mu \, {\mathcal{T}} u$ for every $\mu\le 1$. Assume, on the contrary, that there exists some $\mu\le 1$ such that $u(t)=\mu \, {\mathcal{T}} u(t)$ for all $t\in I$. Then,
		\begin{eqnarray*}
		\sigma\, \|u \| &\le& \int_{0}^{T}u(t)\,dt = \mu \int_{0}^{T} {\mathcal{T}} u(t)\,dt = \mu \int_{0}^{T} \int_{0}^{T} G_P(t,s)\,f(s,u(s))\,ds\,dt\\
		&=&\mu \int_{0}^{T} \left(\int_{0}^{T} G_P(t,s) \,dt\right)f(s,u(s))\,ds \le \mu\, T\, \max_{t,s\in I} \left\{G_P(t,s) \right\} \int_{0}^{T} f(s,u(s))\,ds \\
		&\le& \mu\, T\, \max_{t,s\in I} \left\{G_P(t,s) \right\} \int_{0}^{T} \tilde{f}(s,u(s))\,ds \le \mu\, T\, \max_{t,s\in I} \left\{G_P(t,s) \right\} \int_{0}^{T} \tilde{f}(s,\|u\|)\,ds \\
		&<& \mu\, T^2\, \frac{\eta}{\sigma}\, \frac{\sigma^2}{T^2\, \eta} \|u\| \le \sigma\, \|u \|,
	\end{eqnarray*}
		which is a contradiction. As a consequence, $i_K(T,\Omega_2)=1$.
		
		We conclude that operator ${\mathcal{T}}$ has a fixed point, that is, there exists at least a nontrivial solution of problem \eqref{hill_periodic}.
	\end{proof}
	
		The previous theorem is also valid if the Green's function is nonnegative. In this case, hypothesis $(H_3)$ would be trivially fulfilled and hypothesis $(H_2)$ is not necessary since it is only used to proof that $\mathcal{T}$ maps the cone to itself, which is obvious (since $f$ is nonnegative) when $G_P$ is nonnegative. On the other hand, we would need to add the hypothesis that $f_0=\infty$ and $f^\infty=0$ (which can not be deduced if we eliminate $(H_2)$).
		
		The result is the following
			\begin{corol}\label{cor-GP-posit}
		Assume that $\lambda_P < 0 \le \lambda_A$ and hypothesis $(H_1)$ is fulfilled. Then, if  $f_0=\infty$ and $f^\infty=0$ there exists at least one positive solution of problem \eqref{hill_periodic} in the cone $K$.
	\end{corol}

	\begin{remark}\label{R-int-Graef}
	We note that for a nonnegative Green's function, we generalize the results of Graef, Kong and Wang \cite{gkw1, gkw2} and Webb \cite{webb} since our condition $(H_3)$ is weaker than condition \eqref{int-graef} considered by them.
	\end{remark}
		
\begin{corol}
If $f(t,x)\equiv f(t)\in L^1(I)$ satisfies $(H_2)$, then the unique solution of \eqref{hill_periodic} is a nonnegative function on $[0,T]$.
\end{corol}

	\begin{remark}\label{rem_a}
	We note that $u(t)\equiv 1$ is the unique solution of the periodic problem
	\begin{equation*}\left\{\begin{split} 
	& u''(t)+a(t)\,u(t)=a(t), \quad t\in I, \\
	& u(0)=u(T), \ u'(0)=u'(T).
	\end{split}\right.\end{equation*}
	Therefore it is clear that
	\begin{equation}\label{int-a}
	\int_{0}^{T} G_P(t,s)\,a(s)\,ds=1>0
	\end{equation}
	and so the previous reasoning is also valid if $a\ge 0$, $a>0$ on $[c,d]$, and we change the definition of $\gamma$ by
	\[\gamma^*= \inf_{t\in I} \frac{\int_{0}^{T} G^+_P(t,s)\,a(s)\,ds}{\int_{0}^{T} G^-_P(t,s)\,a(s)\,ds}.\]
	In this case, assumption $(H_2)$ would be substituted by
	\begin{itemize}
	\item[$(H_2^*)$] There exist two positive constants $m$ and $M$ such that $m\,a(t)\le f(t,u) \le M\,a(t)$ for every $t\in I$, $u>0$. Moreover, these constants satisfy that $\frac{M}{m}\le \gamma^*$.
	\end{itemize}
	\end{remark}

	\subsection{Neumann, Dirichlet and Mixed boundary value problems}
		
	From classical spectral theory \cite{zettl}, it is very well know that, as in the periodic case, for any of the boundary conditions introduced on Lemma~\ref{l-Meu-Dir-Mix}, there exists a positive eigenfunction on $(0,T)$ related to the correspondent smallest eigenvalue. Therefore, if we are in the case in which $L[a]$ operator coupled with the associated boundary conditions is nonresonant and the related Green's function changes sign (different cases are characterized on Lemma \ref{l-Meu-Dir-Mix}), we could follow the same argument as in the previous section to define $\gamma$ and we would obtain analogous existence results. Hypothesis $(H_1)-(H_3)$ would be the same with the suitable notation for each of the problems (that is, considering in each case the appropriate Green's function and eigenfunction).
	
	\begin{remark}
	For Neumann problem, it is not difficult to verify that we also have that if $a(t)=\rho^2$ then
	\[\int_{0}^{T} {G_N(t,s)\,ds}=\frac{1}{\rho^2},\]
and condition $(H_3)$ is trivially fulfilled for $[c,d]=I$

	On the other hand, since $u(t)\equiv 1$ is the unique solution of
	\begin{equation*}\begin{split} 
	u''(t)+a(t)\,u(t)=a(t), \; t\in I, \quad
	 u'(0)=u'(T)=0,
	\end{split}\end{equation*}
	Remark \ref{rem_a} is also valid for Neumann problem. 
	
	\end{remark}
	
	\begin{remark}
	For Dirichlet problem, condition $(H_3)$ does not hold for $[c,d]=I$. This is due to the fact that $G_D(t,\cdot)$ satisfies the Dirichlet boundary value conditions for all $t\in[0,T]$, that is, $G_D(t,0)=G_D(t,T)=0$. 
	
	It is important to note that the eigenfunction $v_D$ is positive on $(0,T)$ but $v_D(0)=v_D(T)=0$, so condition $(H_2)$ would imply that $f(0,x)=f(T,x)=0$ for every $x\ge 0$. However, since as we have mentioned, $[c,d]\neq I$, this property does not affect on the fact that $f_0=\infty$.
	
	An analogous situation occurs for Mixed problems. In these cases it is also impossible to consider $[c,d]=I$ since the corresponding Green's functions and eigenfunctions vanish on one side of the interval.
	
	Moreover, if we consider Dirichlet and Mixed problems,  the constant function $u(t)\equiv 1$ is not a solution of the related linear problem $L[a]\, u(t)=a(t)$. So, Remark \ref{rem_a}  is not longer valid for such situations.	
	\end{remark}
	
	\begin{remark}\label{R-int-Graef2}
	As it was commented in Remark \ref{R-int-Graef}, we also generalize the results of Graef, Kong and Wang \cite{gkw1, gkw2} and Webb \cite{webb} for a nonnegative Green's function coupled with Neumann conditions.
		
	Moreover, the results in \cite{gkw1, gkw2, webb} could not be applied to any Dirichlet problem since the related Green's function will cancel on the whole lines $s=0$ and $s=T$ so the minimum in \eqref{int-graef} would be 0, however our result could be applied. The same will happen with any Mixed problem.
		Again, hypothesis $(H_2)$ is not necessary in this case and we would need to add the hypothesis that $f_0=\infty$ and $f^\infty=0$.	
	\end{remark}

	\section{Periodic boundary value problem with constant potential} \label{sect_ex_per}
	This section is devoted to the particular case in which the potential $a$ is constant. As we will see, in this situation it is possible to calculate the exact value of $\gamma$.
	
		It is very well known (see \cite{cabada2, zettl}) that the eigenvalues associated to the periodic problem
		\begin{equation}\label{per-const}
			u''+\rho^2\, u=0, \quad u(0)=u(T),\;u'(0)=u'(T)
		\end{equation}
	are $\lambda_n=(2n\pi/T)^2$ with $n=0,1,2,\dots$. 
	
		The eigenfunctions associated to the first eigenvalue $\lambda_P=0$ are the constants, which can be written as multiples of a representative eigenfunction $v_P(t)\equiv1$.
	
	Moreover, the related Green's function is strictly negative in the square $I \times I$ if and only if $\lambda<0$ and it is nonnegative on $I \times I$ if and only if $0<\lambda \le (\pi/T)^2$ (see \cite{cacilu} for details).
	
			For, $\rho>0$ a nonresonant value, the explicit form of $G_P$ is the following (see \cite{cabada1, cabada2, ma,ZA}):
			\[
				G_P(t,s)=\begin{cases}
									\frac{\sin \rho(t-s)+\sin \rho(T-t+s)}{2\rho(1-\cos \rho\,T)},&
									0\le s\le t\le T, \\[.2cm]
									\frac{\sin \rho(s-t)+\sin \rho(T-s+t)}{2\rho(1-\cos \rho\,T)},&
									0\le t\le s\le T \,.
								\end{cases}
			\]
		From \eqref{int-a} it is clear that
		$$g(t)=\int_0^T G_P(t,s)\,ds=\frac{1}{\rho^2},$$
				therefore we define
		\[
			\gamma=\min_{t\in[0,T]}
			\frac{\int_0^TG_P^+(t,s)\,ds}{\int_0^TG_P
			^-(t,s)\,ds}>1
		\]
		for all $\rho>\pi/T$, $\rho\neq k\pi/T$, $k=1, 2, \ldots$ 
		
		Let us make a careful study of this value $\gamma$. 
		It is very well-known that the Green's function related to the periodic problem \eqref{per-const} satisfies that
		\[G_P(t,s)=G_P(0,t-s) \ \text{ and } \ G_P(t,s)=G_P(T-t,T-s) \]
		(see \cite{cabada2} for the details). Therefore, 
		\begin{equation*}
		\int_0^T G_P(t,s)\,ds= \int_0^t G_P(t,s)\,ds + \int_t^T G_P(t,s)\,ds,
		\end{equation*}
		where
		\begin{equation*}
		\int_0^t G_P(t,s)\,ds= \int_0^t G_P(0,t-s)\,ds = \int_0^t G_P(0,T+s-t)\,ds= \int_{T-t}^T G_P(0,s)\,ds
		\end{equation*}
		and 
		\begin{equation*}
		\int_t^T G_P(t,s)\,ds= \int_t^T G_P(0,T+s-t)\,ds = \int_T^{2T-t} G_P(0,s)\,ds= \int_0^{T-t} G_P(0,s)\,ds,
		\end{equation*}
		that is 
		\begin{equation*}
		\int_0^T G_P(t,s)\,ds= \int_0^T G_P(0,s)\,ds  \quad \forall \, t\in [0,T].
		\end{equation*}
		The same argument is valid for both the positive and the negative parts of $G_P$, that is 
		\begin{equation*}
		\int_0^T G^+_P(t,s)\,ds= \int_0^T G^+_P(0,s)\,ds \ \text{ and } \ \int_0^T G^-_P(t,s)\,ds= \int_0^T G^-_P(0,s)\,ds \quad \forall \, t\in [0,T],
		\end{equation*}
		so the ratio  $\frac{\int_0^T G_P^+(t,s) \,ds}{\int_0^T G_P^-(t,s) \,ds}$ is constant for all $t\in[0,T]$. 
			
			 This implies that we can restrict our analysis to the case $t=0$, that is, to assume that 
				\[
			\gamma=
			\frac{\int_0^TG_P^+(0,s)\,ds}{\int_0^TG_P^-(0,s)\,ds}.
		\]
			
		We have that 
			\[
				G_P(0,s)=\frac{\sin \rho s+ \sin \rho(T-s)}{2\rho(1-\cos\rho T)},
			\]
		so $G_P(0,s)=0$ if and only if $s=\frac{T}{2}+\frac{(2k+1)\pi}{2\rho}$. We will consider four cases:
		\begin{enumerate}
			\item[ \underline{Case 1A:} \hspace*{-1cm}] \hspace*{0.8cm} $G_P(0,\frac{T}{2})\,G_P(0,0)>0$ and $G_P(0,\frac{T}{2})>0$;
			\item[\underline{Case 1B:} \hspace*{-1cm}] \hspace*{0.8cm} $G_P(0,\frac{T}{2})\,G_P(0,0)>0$ and $G_P(0,\frac{T}{2})<0$;
			\item[\underline{Case 2A:} \hspace*{-1cm}] \hspace*{0.8cm} $G_P(0,\frac{T}{2})\,G_P(0,0)<0$ and $G_P(0,\frac{T}{2})>0$;
			\item[\underline{Case 2B:} \hspace*{-1cm}] \hspace*{0.8cm} $G_P(0,\frac{T}{2})\,G_P(0,0)<0$ and $G_P(0,\frac{T}{2})<0$.
		\end{enumerate}
		
		Computing these values, we find that
		
		\begin{itemize}
			\item[] if $\frac{(4k+1)\pi}{T}<\rho<\frac{(4k+2)\pi}{T}$ for some $k\in\N_0$, we are in case 2A and $\gamma=\frac{2k+1}{2k+1-\sin(\rho\,T/2)}$;
			\item[] if $\frac{(4k+2)\pi}{T}<\rho<\frac{(4k+3)\pi}{T}$ for some $k\in\N_0$, we are in case 2B and $\gamma=\frac{2k+1-\sin(\rho\,T/2)}{2k+1}$;
			\item[] if $\frac{(4k-1)\pi}{T}<\rho<\frac{4k\pi}{T}$ for some $k\in\N$, we are in case 1B and $\gamma=\frac{2k}{2k+\sin(\rho\,T/2)}$;
			\item[] if $\frac{4k\pi}{T}<\rho<\frac{(4k+1)\pi}{T}$ for some $k\in\N$, we are in case 1A and $\gamma=\frac{2k+\sin(\rho\,T/2)}{2k}$.
		\end{itemize}
	
	\noindent In the cases where $\rho=(2k+1)\pi$ for some $k\in\N$, the value of $\gamma$ coincides with the limit when $\rho\to 2k+1$. The graph of $\gamma$ for a given value $\rho$ is sketched in Figure \ref{gamma_P}.
	\begin{figure}[H]
	\begin{center} \includegraphics[scale=0.9]{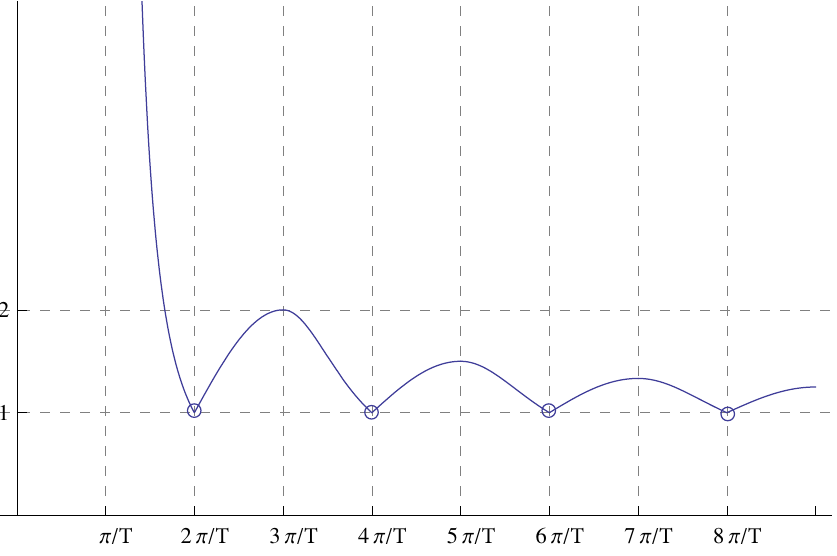}
	\caption{Graph of $\gamma$ for the periodic problem.}\label{gamma_P}
	\end{center}
	\end{figure}

	\section{Dirichlet boundary value problem with constant potential}
	\label{sect-Dir-constant}
	Let us now try to prove some analogue results for  Dirichlet boundary conditions. In this case, the eigenvalues for the Dirichlet problem 
			\[
			u''(t)+\lambda\, u(t)=0, \; \text{ for } t\in(0,T), \quad u(0)=u(T)=0,
		\]
	are $\lambda_n=(n\pi/T)^2$ for $n=1,2,3\cdots$ and it follows easily that the eigenfunctions associated to  $\lambda_D\equiv \lambda_1=(\pi/T)^2$ are the multiples of the function $v_D(t)=\sin (\frac{\pi t}{T})$.
	
	It is very well known that the associated Green's function is strictly negative if and only if $\lambda<\lambda_1=(\pi/T)^2$, and it changes sign for any nonresonant value of $\lambda>(\pi/T)^2$.
	
	 Considering  $\lambda=\rho^2$
	 for $\rho\neq \frac{n\pi}{T}$, with $n\in\N$, we have $\int_0^T G_D(t,s)\sin(\frac{\pi s}{T})\,ds>0$ for $t\in(0,T)$, and we define 
		\[
			\gamma(\rho)=\inf_{t\in(0,T)}\gamma(t,\rho)=
			\inf_{t\in(0,T)}
			\frac{\int_0^TG_D^+(t,s)\sin(\frac{\pi s}{T})\,ds}{\int_0^T G_D
			^-(t,s)\sin(\frac{\pi s}{T})\,ds}.
		\]
	The explicit formula for the Green's function in the nonresonant cases is given by (see \cite{cabada2})
		\begin{equation}
		\label{e-G-Dir}
		G_D(t,s)=\begin{cases}
			G_1(t,s)=-\frac{\sin (\rho s)\,\sin \rho(T-t)}{\rho\sin (\rho\,T)},&
			0\le s\le t\le T, \\[.2cm]
			G_2(t,s)=-\frac{\sin (\rho t)\,\sin \rho(T-s)}{\rho\sin (\rho\,T)},&
			0\le t\le s\le T \,.
		\end{cases}
		\end{equation}
		
		We will consider two cases:
		
		\begin{enumerate}
			\item[\underline{Case 1:} \hspace*{-1cm}] \hspace*{0.8cm} $\frac{(2n-1)\pi}{T}<\rho<\frac{2n\pi}{T}$ for $n\in \N$;
			\item[\underline{Case 2:} \hspace*{-1cm}] \hspace*{0.8cm} $\frac{2n\pi}{T}<\rho<\frac{(2n+1)\pi}{T}$ for $n\in \N$.
		\end{enumerate}
	In case~1 the function $\gamma(t,\rho)$ has a different computation in each of the $4n-1$ intervals
	\begin{equation*}\begin{split}
		\left]0,T-\frac{(2n-1)\pi}{\rho\,T}\right], \, &
		\left[T-\frac{(2n-1)\pi}{\rho\,T},\frac{\pi}{\rho\,T}\right], \, \left[\frac{\pi}{\rho\,T},T-\frac{(2n-2)\pi}{\rho\,T}\right], \, \left[T-\frac{(2n-2)\pi}{\rho\,T},\frac{2\,\pi}{\rho\,T}\right],\\ \cdots &
		\left[\frac{(2n-2)\pi}{\rho\,T},T-\frac{\pi}{\rho\,T}\right], \,\left[T-\frac{\pi}{\rho\,T},\frac{(2n-1)\pi}{\rho\,T}\right],\,
		\left[\frac{(2n-1)\pi}{\rho\,T},T\right[
	\end{split}	\end{equation*}

		and in case~2, it has a different computation in each of the $4n+1$ intervals
		\[
			\left]0,T-\frac{2n\pi}{\rho\,T}\right],\,
			\left[T-\frac{2n\pi}{\rho\,T},\frac{\pi}{\rho\,T}\right], \cdots,
			\left[T-\frac{\pi}{\rho\,T},\frac{2n\pi}{\rho\,T}\right],\,
			\left[\frac{2n\pi}{\rho\,T},T\right[\,.
		\]
	In both cases, given a fixed $\rho$ it is easy to calculate the value of $\gamma(t,\rho)$. However the general expression for an arbitrary $\rho$ requires very long computations which are not fundamental for the purpose of this paper. Because of this, we are going to calculate the general expression of $\gamma(\rho)$ only for the first intervals of $\rho$, in particular for $\rho<\frac{6\,\pi}{T}$.
	
	For $\rho<\frac{6\,\pi}{T}$, we can see that the infimum is attained at $t=0$, so we will restrain our analysis to the first interval of $t$ in both cases in order to obtain the exact expression of $\gamma(\rho)$ for $\rho<\frac{6\,\pi}{T}$. 
	
	In case~1 we have
		\[
			\int_0^T G_D^+(t,s)\sin\left(\frac{\pi s}{T}\right)\,ds=
			\int_{T-\frac{\pi}{\rho}}^T G_2(t,s)\sin\left(\frac{\pi s}{T}\right)\,ds+
			\sum_{i=2}^n\int_{T-\frac{(2i-1)\pi}{\rho\,T}}^{T-\frac{(2i-2)\pi}{\rho\,T}}G_2(t,s)\sin\left(\frac{\pi s}{T}\right)\,ds
		\]
	and
		\begin{align*}
			-\int_0^T G_D^-(t,s)\sin\left(\frac{\pi s}{T}\right)\,ds& =\int_0^t G_1(t,s)\sin\left(\frac{\pi s}{T}\right)\,ds +\int_t^{T-\frac{(2n-1)\pi}{\rho\,T}} G_2(t,s)\sin\left(\frac{\pi s}{T}\right)\,ds\\
			& +
			\sum_{i=1}^{n-1} \int_{T-\frac{2i\pi}{\rho\,T}}^{T-\frac{(2i-1)\pi}{\rho\,T}} G_2(t,s)\sin\left(\frac{\pi s}{T}\right)\,ds=\frac{\sin\left(\frac{\pi t}{T}\right)}{\rho^2-\left(\frac{\pi}{T}\right)^2}-
			\int_0^T G_D^+(t,s)\sin\left(\frac{\pi s}{T}\right)\,ds
		\end{align*}
	so 
		\[
			\gamma(t,\rho)=\frac{\displaystyle\int_{T-\frac{\pi}{\rho\,T}}^T G_2(t,s)\sin\left(\frac{\pi s}{T}\right)\,ds+
			\displaystyle\sum_{i=2}^n \int_{T-\frac{(2i-1)\pi}{\rho\,T}}^{T-\frac{(2i-2)\pi}{\rho\,T}} G_2(t,s) \sin\left(\frac{\pi s}{T}\right)\,ds}
			{\displaystyle\int_{T-\frac{\pi}{\rho\,T}}^T G_2(t,s)\sin\left(\frac{\pi s}{T}\right)\,ds+
			\displaystyle\sum_{i=2}^n\int_{T-\frac{(2i-1)\pi}{\rho\,T}}^{T-\frac{(2i-2)\pi}{\rho\,T}} G_2(t,s)\sin\left(\frac{\pi s}{T}\right)\,ds -\frac{\sin\left(\frac{\pi\,t}{T}\right)}{\rho^2-\left(\frac{\pi}{T}\right)^2}
			}.
	\]
	Doing a similar study for case~2 we get
		\[
			\gamma(t,\rho)=\frac{
			\displaystyle\sum_{i=1}^n\int_{T-\frac{2i\pi}{\rho\,T}}^{T-\frac{(2i-1)\pi}{\rho\,T}}G_2(t,s)\sin\left(\frac{\pi s}{T}\right)\,ds}
			{\displaystyle\sum_{i=1}^n\int_{T-\frac{2i\pi}{\rho\,T}}^{T-\frac{(2i-1)\pi}{\rho\,T}}G_2(t,s)\sin\left(\frac{\pi s}{T}\right)\,ds -\frac{\sin\left(\frac{\pi\,t}{T}\right)}{\rho^2-\left(\frac{\pi}{T}\right)^2}
			}.
	\]
	
	Using the previous expressions it is immediate to calculate $\gamma(t,\rho)$ for any fixed value of $\rho$ and $T$. For instance, computing $\gamma(t,\rho)$ for $T=1$ we obtain:
	\begin{itemize}
		\item If $\rho \in \left(\pi,2\,\pi\right)$, then $\gamma(t,\rho)=\frac{\sin{\rho\,t}\,\sin\frac{\pi^2}{\rho}} {\sin{\rho\,t} \,\sin\frac{\pi^2}{\rho} \,+\, \sin{\rho} \, \sin{\pi\,t}}$.
		\item If $\rho \in \left(2\,\pi,3\,\pi\right)$, then $\gamma(t,\rho)=\frac{\sin\rho\,t\left(\sin\frac{\pi^2}{\rho} \, + \, \sin\frac{2\,\pi^2}{\rho}\right)} {\sin\rho\,t\left(\sin\frac{\pi^2}{\rho} \, + \, \sin\frac{2\,\pi^2}{\rho}\right)  - \, \sin\rho\,\sin\pi\,t}.$
		\item If $\rho \in \left(3\,\pi,4\,\pi\right)$, then $\gamma(t,\rho)=\frac{\sin\rho\,t\, \left(\sin\frac{\pi^2}{\rho} \, + \, \sin\frac{2\,\pi^2}{\rho} \, + \, \sin\frac{3\,\pi^2}{\rho}\right) } {\sin\rho\,t\left(\sin\frac{\pi^2}{\rho} \, + \, \sin\frac{2\,\pi^2}{\rho} \, + \, \sin\frac{3\,\pi^2}{\rho}\right)  + \, \sin\rho\,\sin\pi\,t}.$
		\item If $\rho \in \left(4\,\pi,5\,\pi\right)$, then $\gamma(t,\rho)=\frac{\sin\rho\,t\,\left(\sin\frac{\pi^2}{\rho} \, + \, \sin\frac{2\,\pi^2}{\rho} \, + \, \sin\frac{3\,\pi^2}{\rho} \,+ \, \sin\frac{4\,\pi^2}{\rho} \right)}{\sin\rho\,t\,\left(\sin\frac{\pi^2}{\rho} \, + \, \sin\frac{2\,\pi^2}{\rho} \, + \, \sin\frac{3\,\pi^2}{\rho} \,+ \, \sin\frac{4\,\pi^2}{\rho} \right)-\sin\rho \, \sin\pi\,t}.$
		\item If $\rho \in \left(5\,\pi,6\,\pi\right)$, then $\gamma(t,\rho) =\frac{\sin\rho\,t \, \left(\sin\frac{2\,\pi^2}{\rho} \, + \, \sin\frac{3\,\pi^2}{\rho} \,+ \, \sin\frac{4\,\pi^2}{\rho} \,+ \, \sin\frac{5\,\pi^2}{\rho}\right) \,+\, 2\left(1\,-\,\frac{\pi^2}{\rho^2}\right) \sin\rho\,t } {\sin\rho\,t\,\left(\sin\frac{2\,\pi^2}{\rho} \, + \, \sin\frac{3\,\pi^2}{\rho} \,+ \, \sin\frac{4\,\pi^2}{\rho} \,+ \, \sin\frac{5\,\pi^2}{\rho}\right)  \,+\, \sin\rho\,\sin\pi\,t \,+\, 2\left(1\,-\,\frac{\pi^2}{\rho^2}\right)\sin\rho\,t}$.
	\end{itemize}
	
	In Figure \ref{gamma_t} we have a sketch of the function $\gamma(t,10.8)$ for $T=1$.
	\begin{figure}[H]
		\begin{center} \includegraphics[scale=0.7]{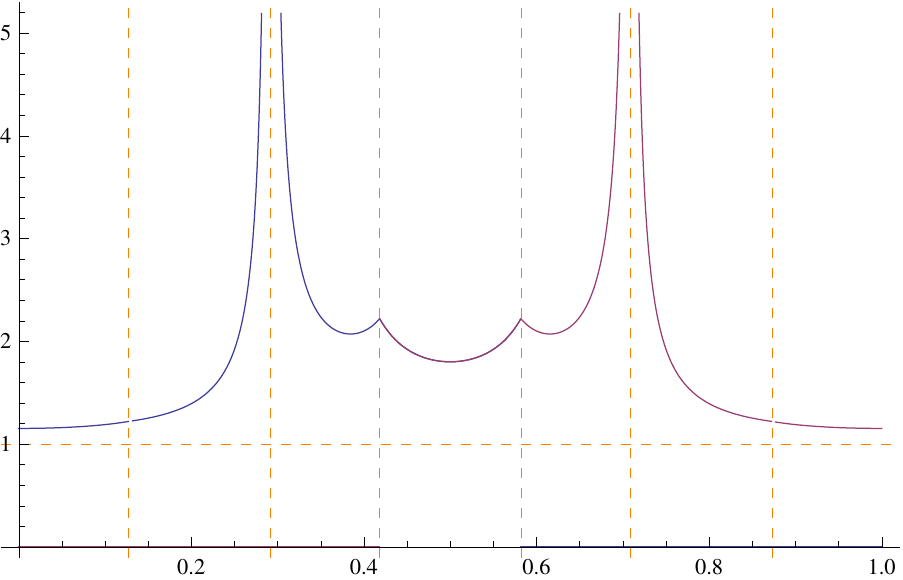}
		\caption{Graph of $\gamma(t,10.8)$ for Dirichlet problem.}\label{gamma_t}
		\end{center}
	\end{figure}

	Computing the limit 
		\[
			\gamma(\rho)=\lim_{t\to 0}\gamma(t,\rho),
		\]
	we get the following expressions for $\gamma(\rho)$:
	\begin{itemize}
		\item If $\rho \in \left(\pi,2\,\pi\right)$, then $\gamma(\rho)=1-\frac{\pi\,\sin{\rho}}{\pi\,\sin{\rho} \,+\, \rho \,\sin\frac{\pi^2}{\rho}}$.
		\item If $\rho \in \left(2\,\pi,3\,\pi,\right)$, then $\gamma(\rho)=1+\frac{\pi\, \sin\rho}{-\pi\,\sin\rho \,+\, \rho\left(\sin\frac{\pi^2}{\rho} \, + \, \sin\frac{2\,\pi^2}{\rho}\right)}.$
		\item If $\rho \in \left(3\,\pi,4\,\pi\right)$, then $\gamma(\rho)=1-\frac{\pi\, \sin\rho}{\pi\,\sin\rho \,+\, \rho\left(\sin\frac{\pi^2}{\rho} \, + \, \sin\frac{2\,\pi^2}{\rho}\, + \, \sin\frac{3\,\pi^2}{\rho}\right)}.$
		\item If $\rho \in \left(4\,\pi,5\,\pi\right)$, then $\gamma(\rho)=1+\frac{\pi\, \sin\rho}{-\pi\,\sin\rho \,+\, \rho\left(\sin\frac{\pi^2}{\rho} \, + \, \sin\frac{2\,\pi^2}{\rho}\, + \, \sin\frac{3\,\pi^2}{\rho}\, + \, \sin\frac{4\,\pi^2}{\rho}\right)}.$
		\item If $\rho \in \left(5\,\pi,6\,\pi\right)$, then $\gamma(\rho)=1 -\frac{\pi\, \sin\rho}{\pi\,\sin\rho \,+\, \rho \left(\sin\frac{\pi^2}{\rho} \, + \, \sin\frac{2\,\pi^2}{\rho}\, + \, \sin\frac{3\,\pi^2}{\rho}\, + \, \sin\frac{4\,\pi^2}{\rho}\, + \, \sin\frac{5\,\pi^2}{\rho}\right) \,+\,2\frac{\rho^2-\pi^2}{\rho}}.$
	\end{itemize}	
	Graphically the function $\gamma(\rho)$ is represented in Figure \ref{gamma_D} for $T=1$.
	\begin{figure}[H]
	\begin{center} 
	\includegraphics[scale=0.7]{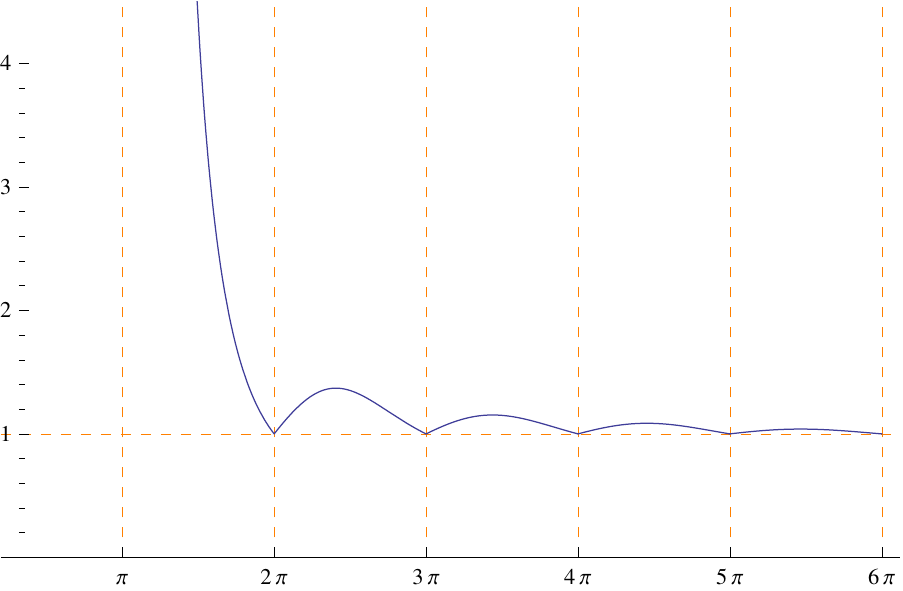}
	\caption{Graph of $\gamma$ for Dirichlet problem.}\label{gamma_D}	
	\end{center}
	\end{figure}

	Let us now see some examples.
	
	\begin{example}
	The Dirichlet BVP
	\begin{equation}\label{ex-D-1}
		u''(t)+60 u(t)=t(1-t) , \text{ for } t\in(0,1)\qquad
			u(0)=u(1)=0 
	\end{equation}
	has a positive solution, since $\gamma(\sqrt{60})\approx 1.36>\frac{4}{3}$
	and $\frac{3\sin(\pi t)}{4\pi}\le t(1-t)\le \frac{\sin(\pi t)}{\pi}$,
	but the solution of the Dirichlet BVP
	\begin{equation}\label{ex-D-2}
		u''(t)+60 u(t)=t , \text{ for } t\in(0,1)\qquad
			u(0)=u(1)=0 
	\end{equation}
	changes sign. We can see the respective solutions in Figures \ref{fig-ex-D-1} and \ref{fig-ex-D-2}.
	
	\begin{multicols}{2}
	\begin{center}
			\begin{figure}[H]
					\begin{center}
		 \includegraphics[scale=0.6]{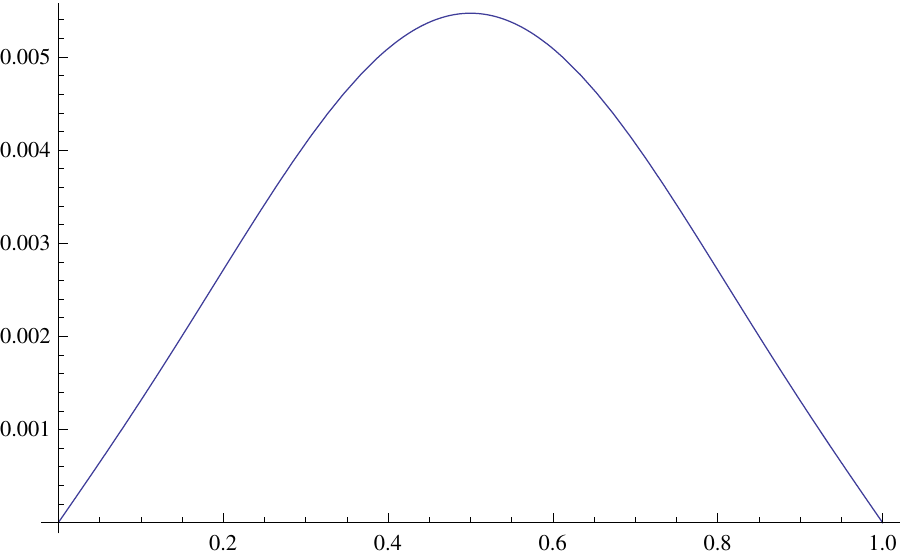}\hspace{1cm}
		\caption{Solution of problem \eqref{ex-D-1}} \label{fig-ex-D-1}
		  \end{center}
	\end{figure}
\columnbreak
			\begin{figure}[H]
		\includegraphics[scale=0.6]{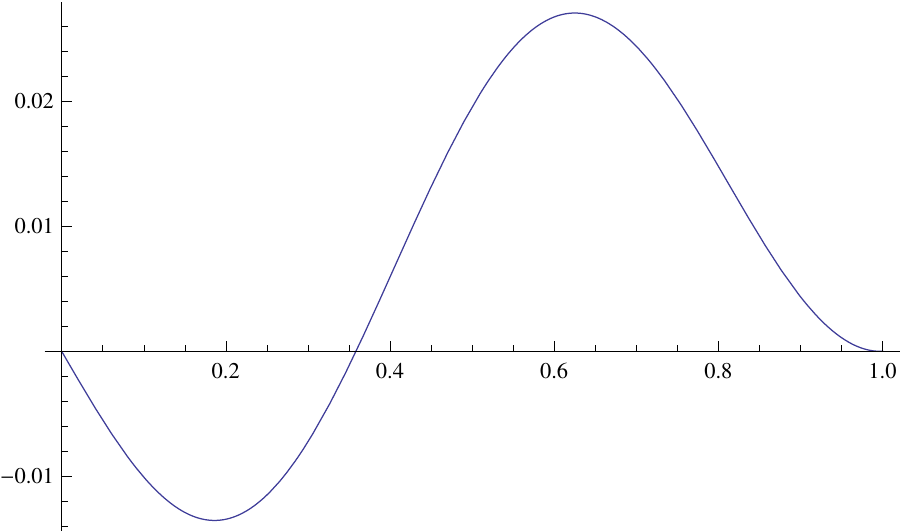}
		\caption{Solution of problem \eqref{ex-D-2}} \label{fig-ex-D-2}
	\end{figure}
  \end{center}
\end{multicols}
\end{example}

\begin{remark}
Analogous arguments and calculations can be done for the Neumann and Mixed problems. 
\end{remark}


\begin{thebibliography}{99}

\bibitem{amann2}  Amann, H. {\em Fixed point equations and Nonlinear Problems in Ordered Banach Spaces},
SIAM Review, Vol. 18,  4 (1976), pp. 620--709.


\bibitem{cabada1} 
{A. Cabada}, {\em The method of lower and upper solutions for second, third, fourth, and higher order boundary value problems}. J. Math. Anal. Appl. 185 (1994), 2, 302--320.

\bibitem{cabada2} 
{A. Cabada}, {\em Green's functions in the theory of ordinary differential equations. SpringerBriefs in Mathematics}. Springer, New York, 2014.


\bibitem{cacidAAA} { A. Cabada; J.\,A. Cid}, {\em Existence and multiplicity of solutions for a periodic Hill's equation with parametric dependence and singularities}, Abstr. Appl. Anal. (2011). 

\bibitem{cacid} { A. Cabada, J.\,A. Cid}, {\em On comparison principles for the periodic Hill's equation}, J. Lond. Math. Soc. (2) 86 (2012), 1, 272--290. 

\bibitem{cacilu} { A. Cabada, J.\,A. Cid, L. L\'opez-Somoza}, {\em Green's functions and spectral theory for the Hill's equation}, Appl. Math. and Comp. 286 (2016), 88--105.

\bibitem{cabinftoj} {A. Cabada, G. Infante, F.\,A.\,F. Tojo}, {\em Nontrivial solutions of Hammerstein integral equations with reflections}, Bound. Value Prob. 2013, \textbf{2013:86}.

\bibitem{gkw1} J. Graef, L. Kong, H. Wang,
{\em A periodic boundary value problem with vanishing Green's function},
Applied Mathematics Letters \textbf{21} (2008), 176--180.

\bibitem{gkw2} J. Graef, L. Kong, H. Wang, \emph{Existence, multiplicity, and dependence on a
parameter for a periodic boundary value problem}, J.
Differential Equations \textbf{245} (2008), 1185--1197.

\bibitem{infpietoj} {G. Infante, P. Pietramala, F.\,A.\,F. Tojo}, {\em Nontrivial solutions of local and nonlocal Neumann boundary value problems}, Proc. Roy. Soc. Edinburgh, \textbf{146A}, 337-369, 2016.

\bibitem{ma} R. Ma, \emph{Nonlinear periodic boundary value problems with sign-changing Green's function}, Nonlinear Analysis \textbf{74} (2011), 1714--1720.

\bibitem{torres}
{ P. Torres},
{\em Existence of one-signed periodic solutions of some second-order differential equations via a Krasnoselskii fixed point theorem},
J. Differential Equations 190 (2003), 2, 643-662.	

\bibitem{webb}
{ J. Webb},
{\em Boundary value problems with vanishing Green's function},
Communications in Applied Analysis 13 (2009), nº 4, 587-596.	

\bibitem{zettl}  A. Zettl, {\em Sturm-Liouville theory.} Mathematical Surveys and Monographs, 121. American Mathematical Society, Providence, RI, 2005.

		
\bibitem{zhang} 
{ M. Zhang}, {\em Optimal conditions for maximum and anti-maximum principles of the periodic solution problem}, Boundary Value Problems, Volume 2010, Article ID 410986,
26 pages doi:10.1155/2010/410986.

\bibitem {ZA} {S. Zhong, Y. An},
{\em  Existence of positive solutions to periodic boundary value problems with sign-changing Green's function}, 
Boundary Value Problems { 2011} (2011) DOI: 10.1186/1687-2770-2011-8.

	
	\end{thebibliography}
\end{document}